%% file: ETCFT_Paper.tex
\documentclass[11pt,a4paper,british]{article}
\usepackage{yfonts, amsmath, amscd, amsthm, latexsym, anysize, graphicx, makeidx, appendix, type1cm, eso-pic, color, wasysym, float, bbold, dsfont, xcolor, framed, mdframed, pzccal, anyfontsize, stmaryrd, euscript} %MnSymbol, amssymb, courier
\usepackage[colorlinks, pagebackref, pdfusetitle, urlcolor=black, citecolor=black, linkcolor=black, bookmarksnumbered, plainpages=false]{hyperref}
\usepackage{yfonts, amsmath, amscd, amsthm, amssymb, latexsym, anysize, graphicx, makeidx, appendix, type1cm, eso-pic, color, wasysym, float, xcolor, framed, mdframed, pzccal, anyfontsize, courier, babel, titlesec, dsfont}
\usepackage [all,pdf,cmtip]{xy}
\SetSymbolFont{largesymbols}{normal}{OMX}{iwona}{m}{n}
\usepackage[T1]{fontenc}

\usepackage{mathpazo} %mathpazo, cmbright
\makeindex
\marginsize{2.5cm}{2.5cm}{2cm}{2cm}
\graphicspath{{/Users/Ramses/Documents/Universitat/02_Thesis/Pictures/}} % Specifies the directory where pictures are stored

\makeatletter
\def\th@plain{\thm@notefont{} \itshape}
\def\th@definition{\thm@notefont{}\normalfont}
\makeatother

\DeclareFontFamily{U}{futm}{}
\DeclareFontShape{U}{futm}{m}{n}{
  <-> s * [1] fourier-bb
  }{}
\DeclareSymbolFont{Ufutm}{U}{futm}{m}{n}
\DeclareSymbolFontAlphabet{\mathbb}{Ufutm}
\mathchardef\mhyphen="2D % Definim un gui— curt.

\makeatletter % Definim un barret cap per avall
\DeclareRobustCommand\widecheck[1]{{\mathpalette\@widecheck{#1}}}
\def\@widecheck#1#2{%
    \setbox\z@\hbox{\m@th$#1#2$}%
    \setbox\tw@\hbox{\m@th$#1%
       \widehat{%
          \vrule\@width\z@\@height\ht\z@
          \vrule\@height\z@\@width\wd\z@}$}%
    \dp\tw@-\ht\z@
    \@tempdima\ht\z@ \advance\@tempdima2\ht\tw@ \divide\@tempdima\thr@@
    \setbox\tw@\hbox{%
       \raise\@tempdima\hbox{\scalebox{1}[-1]{\lower\@tempdima\box
\tw@}}}%
    {\ooalign{\box\tw@ \cr \box\z@}}}
\makeatother

\renewcommand{\thefootnote}{\fnsymbol{footnote}}

\newtheorem{defn0}{Definition}[section]
\newtheorem{prop0}[defn0]{Proposition}
\newtheorem{thm0}[defn0]{Theorem}
\newtheorem{lemma0}[defn0]{Lemma}
\newtheorem{corollary0}[defn0]{Corollary}
\newtheorem{example0}[defn0]{Example}
\newtheorem{remark0}[defn0]{Remark}
\newtheorem{conjecture0}[defn0]{Conjecture}

\newenvironment{proposition}{\smallskip\begin{prop0}}{\end{prop0}}
\newenvironment{theorem}{\smallskip\begin{thm0}}{\end{thm0}}
\newenvironment{lemma}{\smallskip\begin{lemma0}}{\end{lemma0}}
\newenvironment{corollary}{\smallskip\begin{corollary0}}{\end{corollary0}}
\newenvironment{example}{\smallskip\begin{example0}}{\end{example0}}
\newenvironment{remark}{\smallskip\begin{remark0}}{\end{remark0}}
\newcommand{\CC}{\operatorname{\mathfrak{C}}}
\newcommand{\DD}{\operatorname{\mathfrak{D}}}
\newcommand{\EE}{\operatorname{\mathfrak{E}}}
\newcommand{\FF}{\operatorname{\mathfrak{F}}}
\newcommand{\GG}{\operatorname{\mathfrak{G}}}

\newcommand{\KK}{\operatorname{\mathfrak{K}}}

\newcommand{\MM}{\operatorname{\mathfrak{M}}}

\newcommand{\PP}{\operatorname{\mathfrak{P}}}

\newcommand{\ob}{\operatorname{Obj}}
\newcommand{\comp}{\operatorname{\mathpzc{Comp}}}

\newcommand{\cat}{\operatorname{\mathpzc{Cat}}}
\newcommand{\topo}{\operatorname{\mathpzc{Top}}}

\newcommand{\categ}{\operatorname{2\mathpzc{Cat}}}
\newcommand{\surf}{\operatorname{\mathpzc{Surf}}}
\newcommand{\disc}{\operatorname{\mathpzc{Disc}}}
\newcommand{\grap}{\operatorname{\mathpzc{Graph}}}
\newcommand{\forest}{\operatorname{\mathpzc{Forest}}}

\newcommand{\aaa}{\operatorname{\mathpzc{A}}}
\newcommand{\bbb}{\operatorname{\mathpzc{B}}}
\newcommand{\ccc}{\operatorname{\widetilde{\mathpzc{C}}}}

\newcommand{\gcc}{\operatorname{\mathpzc{C}}}

\newcommand{\mmm}{\operatorname{\mathpzc{M}}}
\newcommand{\nnn}{\operatorname{\mathpzc{N}}}

\newcommand{\goo}{\operatorname{\mathpzc{O}}}

\newcommand{\homo}{\operatorname{Hom}}

\newcommand{\aut}{\operatorname{Aut}}

\newcommand{\id}{\operatorname{Id}}

\newcommand{\K}{\operatorname{\mathbb{K}}}

\newcommand{\rad}{\operatorname{rad}}

\newcommand{\tr}{\operatorname{Tr}}

\renewcommand{\thefootnote}{\fnsymbol{footnote}}

\newcommand{\mo}{\operatorname{Mod}}
\DeclareMathOperator*{\hoc}{hocolim}

\begin{document}

\title{On the Structure of Open Equivariant Topological Conformal Field Theories}
\author{Rams\`es Fern\`andez-Val\`encia}
\date{}

\maketitle

%=========================================================================================
%=========================================================================================

\begin{abstract}
\noindent
We give a classification of open equivariant topological conformal field theories in terms of Calabi-Yau $A_\infty$-categories endowed with a group action. 
\end{abstract}

\tableofcontents

\let\thefootnote\relax\footnotetext{\date{\today}}
\let\thefootnote\relax\footnotetext{Email: ramses.fernandez.valencia@gmail.com}

%%%%%%%%%%%%%%%%%%%%%%%%%%%%%%%%%%%%%%%%%%%%%%%%%%%%%%%%%%%%%
%%%%%%%%%%%%%%%%%%%%%%%%%%%%%%%%%%%%%%%%%%%%%%%%%%%%%%%%%%%%%
\section{Introduction}
%%%%%%%%%%%%%%%%%%%%%%%%%%%%%%%%%%%%%%%%%%%%%%%%%%%%%%%%%%%%%
%%%%%%%%%%%%%%%%%%%%%%%%%%%%%%%%%%%%%%%%%%%%%%%%%%%%%%%%%%%%%

We can think of 2-dimensional topological conformal field theories (TCFT for short) as an extension of topological quantum field theories in dimension 2 (henceforth TQFT) since we work not with surfaces but with Riemann surfaces, that is: we consider surfaces equipped with a complex structure. The first studies around TCFTs were made by Moore and Segal, in \cite{Moore00, MoSe06}, who first gave a precise definition for TCFT and suggested the importance of its study.
\newpage
For a finite set, whose elements are called $D$-branes, let us consider $\mathpzc{OC}_{\Lambda}$ the category whose class of objects are 1-manifolds (disjoint unions of circles and intervals) with boundary labelled by $D$-branes and with class of morphisms given by given by singular chains on moduli spaces of Riemann surfaces. Given a field $\K$ of characteristic zero, it makes sense to consider differential graded symmetric monoidal functors of the form $\FF:\mathpzc{OC}_{\Lambda}\to\comp_{\K}$, where $\comp_{\K}$ is the category of chain complexes over $\K$. Such a functor $\FF$, satisfying certain conditions, is called a \index{TCFT} 2-dimensional \textit{topological conformal field theory}. Depending on the boundary components of the Riemann surfaces we work with, we can talk about open, closed or open-closed TCFTs.
\\\\
Relying on a ribbon graph decomposition of the moduli space of marked Riemann surfaces, open TCFTs were classified by Costello \cite{Costello07} in terms of $A_\infty$-categories satisfying a Calabi-Yau condition. Costello also gave a universal extension from open TCFTs to open-closed TCFTs and proves that the homology associated to the closed part of an open-closed TCFT is described in terms of the Hochschild homology of the Calabi-Yau $A_\infty$-category associated to its open part. 
\\\\
Inspired by Costello and Braun \cite{thbraun12}, the author developed a full generalization of Costello's results in the unoriented setting, that is replacing Riemann surfaces with Klein surfaces \cite{Ramses15, thramses15}. 
\\\\
If we consider now a discrete group $G$ and moduli spaces of geometric structures over Riemann surfaces, we can talk about \textit{equivariant topological conformal field theories}. As before, depending on the boundary components of our Riemann surfaces, we can talk about open, closed or open-closed equivariant TCFTs. By extending the techniques developed in \cite{Ramses15} to the equivariant setting, the research developed here studies the classification of open equivariant TCFTs. The main results are:
\begin{theorem}\label{thetheorem}
For a discrete group $G$, there is a homotopy equivalence between open equivariant TCFTs and Calabi-Yau $A_\infty$-categories endowed with a group action.
\end{theorem}

\begin{proposition}\label{quasi-equivalence}
The following categories are quasi-equivalent:
\begin{enumerate}
\item The category of unital extended equivariant $A_\infty$-categories;
\item the category of unital equivariant $A_\infty$-categories and
\item the category of unital equivariant DG categories.
\end{enumerate}
\end{proposition}

\section{Moduli spaces of geometric structures}

%===========================================================================================
%===========================================================================================

An important ingredient of our classification are moduli spaces of Riemann surfaces equipped with a principal $G$-bundle. This section, whose main reference is \cite{Giansiracusa14}, introduces the main definitions and results related to moduli spaces of geometric structures on surfaces. We also introduce open equivariant topological conformal field theory and the categories governing such theories together with its description in terms of generators and relations, which will lead us to the equivalence between open equivariant theories and Calabi-Yau categories.

%%%%%%%%%%%%%%%%%%%%%%%%%%%%%%%%%%%%%%%%%%%%%%%%%%%%%%%%%%%%
%\subsection{Moduli spaces of geometric structures}
%%%%%%%%%%%%%%%%%%%%%%%%%%%%%%%%%%%%%%%%%%%%%%%%%%%%%%%%%%%%

%\textit{An important ingredient of our classification are moduli spaces of Riemann surfaces equipped with a principal $G$-bundle. This section, whose main reference is \cite{Giansiracusa14}, goes further and introduces the main definitions and results related to moduli spaces of geometric structures on surfaces.}

%%%%%%%%%%%%%%%%%%%%%%%%%%%%%%%%%%%%%%%%%%%%%%%%%%%%%%%%%%%%
\subsection{Strict 2-categories and functors} %%%%%%%%%%%%%%%%%%%%%%%%%%%%%%%%%%%%%
%%%%%%%%%%%%%%%%%%%%%%%%%%%%%%%%%%%%%%%%%%%%%%%%%%%%%%%%%%%%

A \index{Strict 2-category} \textit{strict 2-category} is a category $\aaa$ enriched over $\cat$, the category of categories and functors, what means that it has:
\begin{enumerate}
\item A class of objects $\ob(\aaa)$;
\item for objects $a_1,a_2\in\ob(\aaa)$, a category $\homo_{\aaa}(a_1,a_2)$;
\item composition functors $\homo_{\aaa}(a_1,a_2)\times\homo_{\aaa}(a_2,a_3)\to\homo_{\aaa}(a_1,a_3)$ which are strictly associative and have a unit in $\homo_{\aaa}(a,a)$.
\end{enumerate}

The objects of the category $\homo_{\aaa}$ are called 1-morphisms whilst the morphisms are called 2-morphisms. We assume that all 2-morphisms are isomorphisms.

\begin{remark}
We shall denote by $\categ$ the strict 2-category whose objects are small categories and whose $\homo$ categories are the categories of functors and natural transformations.
\end{remark}

For 2-categories $\aaa$ and $\bbb$, a \index{2-functor} \textit{2-functor} $\FF:\aaa\to \bbb$ is a map $F:\ob(\aaa)\to \ob(\bbb)$ together with a functor $\homo_{\aaa}(a_1,a_2)\to\homo_{\bbb}(F(a_1),F(a_2))$ such that all these functors are strictly compatible with the composition functors.

%%%%%%%%%%%%%%%%%%%%%%%%%%%%%%%%%%%%%%%%%%%%%%%%%%%%%%%%%%%%
\subsection{Surfaces with collars} %%%%%%%%%%%%%%%%%%%%%%%%%%%%%%%%%%%%%%%%%%%
%%%%%%%%%%%%%%%%%%%%%%%%%%%%%%%%%%%%%%%%%%%%%%%%%%%%%%%%%%%%

Let $\Sigma$ be a surface with boundary partitioned into intervals and locally isomorphic to the cartesian product $[0,+\infty)\times [0,+\infty)$. For a boundary interval $J\subset \Sigma$, a \index{Collar}\textit{collar} of $J$ is a smooth embedding $\varpi$ of $(-1,0]\times[0,1]$ into $\Sigma$ such that:
\begin{enumerate}
\item It preserves boundaries and
\item is a diffeomorphism of $\{0\}\times [0,1]$ onto $J$. 
\end{enumerate}
A surface with a finite set of disjoint boundary intervals equipped with disjoint collars labelled by a set $I$, is called \index{$I$-collared surface} \textit{$I$-collared}. A diffeomorphism of $I$-collared surfaces $\Sigma_1\to\Sigma_2$ is a diffeomorphism of the underlying surfaces that respects the labelling and parametrization of the collars.
\\\\
For a surface $\Sigma$ with disjoint boundary intervals $J_1,J_2$ equipped with disjoint collars $\varpi_1,\varpi_2$ respectively, we can glue these two boundary intervals together and obtain a new smooth surface as follows: let $\Sigma'=\Sigma-(J_1\cup J_2)/\sim$, where $\sim$ describes the identification of $\varpi_1(x)$ with $\varpi_2(x)$ for each  $x\in(-1,0)\times [0,1]$.
\\\\
Collared surfaces and open embeddings form a category which will be denoted by $\surf$. Let $\disc\subset \surf$ denoted the full subcategory whose objects are disjoint unions of discs each having at least one collared boundary interval. We denote with $\widetilde{\surf}$ the category enriched over $\topo$ of finite type surfaces (possibly with boundary and locally modeled on $[0,+\infty)\times [0,+\infty)$) and open embeddings.

%%%%%%%%%%%%%%%%%%%%%%%%%%%%%%%%%%%%%%%%%%%%%%%%%%%%%%%%%%%
\subsection{Sheaves of geometric structures} %%%%%%%%%%%%%%%%%%%%%%%%%%%%%%%%%%%%
%%%%%%%%%%%%%%%%%%%%%%%%%%%%%%%%%%%%%%%%%%%%%%%%%%%%%%%%%%%

A \index{Smooth sheaf} \textit{smooth sheaf} $\varsigma$ on $\widetilde{\surf}$ is an enriched functor $\widetilde{\surf}\to\topo$ sending pushout squares to homotopy pullback squares.
\begin{example}
Let us denote with $\varsigma(\Sigma)$ the space of geometric structures of a given type on an $I$-collared surface $\Sigma$. Some examples of such spaces are: 
\begin{enumerate}
\item The set of orientations on $\Sigma$. 
\item The space of principal $G$-bundles over $\Sigma$, which is the space of maps $\Sigma\to BG$. %The latter example play a central role in this section.
\end{enumerate}
\end{example}
A smooth sheaf $\varsigma$ is \index{Connected sheaf} \textit{connected} if $\varsigma((-1,0)\times I)$ is connected. If $\varsigma$ is connected, then we can choose a base-point $p\in \varsigma((-1,0)\times I)$. Let us consider $J\subset \partial\Sigma$ a boundary interval equipped with a collar $\varpi$. We say that a section $\sigma\in\varsigma(\Sigma)$ is trivial at $J$ if the pullback $\sigma\circ\varpi$ restricts to the chosen base-point $p$ on $(-1,0)\times I$. If a surface $\Sigma$ is endowed with a collection of disjoint collared boundary intervals $\{J_1\dots, J_n\}$, we write $\widetilde{\varsigma}(\Sigma)\subset \varsigma(\Sigma)$ for the subspace consisting of section that are trivial at $J_i$ for each $i\in\{1,\dots, n\}$.
\\\\
Let us denote with $A_\varsigma$ the space $\widetilde{\varsigma}(I\times I)$ of sections that are trivial at each side of the square $I\times I$, which we assume equipped with a collar at each of the intervals $\{0\}\times I$ and $\{1\}\times I$.

\begin{proposition}[Proposition 4.3.1 \cite{Giansiracusa14}]
Gluing squares side to side endows $A_\varsigma$ with an $A_\infty$ composition making it into a group-like $A_\infty$-monoid; the homotopy inverse map is induced by rotating the square $\pi$ $\rad$. Fixing a collared boundary interval $J$ on a surface $\Sigma$, there is a right $A_\infty$-action of $A_\varsigma$ on $\widetilde{\varsigma}(\Sigma)$ by gluing the right side of a square to $J$, and a left $A_\infty$-action given by gluing the left edge of the square to $J$.  
\end{proposition}

\begin{proposition}[Proposition 4.3.2 \cite{Giansiracusa14}]
Let $J_1$ and $J_2$ be two disjointly collared boundary intervals on a surface $\Sigma$, and let $\Sigma'$ be the result of gluing $J_1$ to $J_2$. There is a homotopy equivalence $\varsigma(\Sigma')\simeq \widetilde{\varsigma}(\Sigma)_{hA_\varsigma}$ where the action of $A_\varsigma$ is given, for a square with $\varsigma$-structure $K\in A_\varsigma$, by gluing the left edge of one copy of $K$ to $J_1$ and gluing the left edge of a second copy of $K$ to $J_2$.
\end{proposition}

%%%%%%%%%%%%%%%%%%%%%%%%%%%%%%%%%%%%%%%%%%%%%%%%%%%%%%%%%%%
\subsection{The modular operad of moduli spaces} %%%%%%%%%%%%%%%%%%%%%%%%%%%%%%%%%%
%%%%%%%%%%%%%%%%%%%%%%%%%%%%%%%%%%%%%%%%%%%%%%%%%%%%%%%%%%%

It is convenient to recall the concepts of cyclic and modular operads. We start off with cyclic operads, which were introduced by Getzler, Kapranov and Kontsevich in order to handle algebras endowed with inner products $\langle -, -\rangle$ satisfying $\langle a\cdot b, c\rangle=\langle a, b\cdot c\rangle$ in the associative case and $\langle[a,b], c \rangle=\langle a, [b,c] \rangle$ in the Lie case (Section 1.14 \cite{MaShSt02}). A \index{Cyclic operad} \textit{cyclic operad} in a symmetric monoidal category $\gcc$ is given by a functor $\PP$ from the category of non-empty sets and bijections to $\gcc$ and for each pair of finite sets $I,J$ with elements $i\in I$ and $j\in J$ a composition morphism
$$\circ_{i,j}:\PP(I)\otimes \PP(J)\to \PP(I\sqcup J-\{i,j\})$$
natural in $I$ and $J$ and satisfying the following axioms:
\begin{enumerate}
\item Given finite sets $I,J,K$ and elements $i\in I$, $j_1,j_2\in J$ and $k\in K$, the following diagram commutes:
\[
\xymatrix{
\PP(I)\otimes \PP(J)\otimes \PP(K) \ar[r]^-{\id\otimes\circ_{j_2,k}} \ar[d]_-{\circ_{i,j_1}\otimes\id} & \PP(I)\otimes \PP(J\sqcup K-\{j_2,k\}) \ar[d]^-{\circ_{i,j_1}} \\
\PP(I\sqcup J-\{i,j_1\})\otimes \PP(K) \ar[r]_-{\circ_{j_2,k}} & \PP(I\sqcup J\sqcup K-\{i,j_1,j_2,k\})
}
\]
\item Given finite sets $I,J$ and elements $i\in I$ and $j\in J$ the following diagram commutes:
\[
\xymatrix{
\PP(I)\otimes \PP(J) \ar[r]^-{\circ_{i,j}} \ar[d]_-{(!)} & \PP(I\sqcup J-\{i,j\}) \ar[d]^-{\PP(\tau)} \\
\PP(J)\otimes \PP(I)\ar[r]_-{\circ_{j,i}} & \PP(J\sqcup I-\{j,i\})
}
\]
where $\tau$ denotes the bijection $I\sqcup J-\{i,j\}\cong J\sqcup I-\{j,i\}$ and (!) denotes a swapping.
\item For each set $A=\{a,b\}$ of cardinality 2 there is a morphism $u_A:1_{\gcc}\to\PP(A)$, where $1_{\gcc}$ denotes the monoidal unit, that is natural in $A$ and such that for any finite set and an element $i\in I$ the following diagram commutes:
\[
\xymatrix{
1_{\gcc}\otimes \PP(I) \ar[r]^-{u_A\otimes \id} \ar[dr]_-{\PP(\beta)}& \PP(A)\otimes \PP(I)\ar[d]^-{\circ_{a,i}} \\
 & \PP(A\sqcup I-\{a, i\})
}
\]
where $\beta$ is the canoncial bijection $I\cong A\sqcup I-\{a,i\}$ sending $i$ to $b$.
\end{enumerate}

We define a \index{Modular operad} \textit{modular operad} in $\gcc$ as a cyclic operad $\PP$ together with natural self-composition maps $\circ_{i,j}:\PP(I)\to\PP(I-\{i,j\})$ that commute with the cyclic operad composition maps and with each other.
\begin{example}
The mother of modular operads is given by the Deligne-Mumford moduli spaces of stable curves of genus $g$ with $n+1$ marked points, where the composition maps are defined by interesecting along their marked points.
\end{example}
Given a cyclic operad $\mathfrak{O}$ we define its \index{Modular envelope} \textit{modular envelope} as the modular operad freely generated by $\mathfrak{O}$.
\\\\
It will also be convenient to recall some basics from graph theory. Let us remind that a (finte) graph $\gamma$ is given by a (finite) set of vertices $V(\gamma)$, a (finite) set of half-edges $H(\gamma)$, a map $\lambda: H(\gamma)\to V(\gamma)$ and an involution $\iota: H(\gamma)\to H(\gamma)$. For a graph $\gamma$, we say that two half-edges $a,b$ form an edge if $\iota(a)=b$, a half-edge $a$ is connected to a vertex $v$ if $\lambda(a)=v$. A leg in $\gamma$ is a univalent vertex. A corolla is a graph consisting of a single vertex with several legs connected to it.
\\\\
The topological realization of a graph $\gamma$ is a 1-dimensional CW complex. We call tree a graph whose topological realization is contractible; a forest is a union of trees. 
\\\\
Let us denote $\grap$ the symmetric monoidal category whose objects are disjoint unions of corollas and where the arrows are given by graphs; the latter meaning, intuitively, that we may think of such a morphism as assembling a bunch of corollas to create a graph followed by contracting all the edges so the result is again a bunch of corollas. There is a subcategory $\forest\subset\grap$ whose objects are those of $\grap$ with no isolated vertices and whose arrows are given by forests.
\\\\
If $2\disc$ and $2\surf$ are the strict 2-categories with the sames objects as $\disc$ and $\surf$, respectively, and with space of diffeomorphisms given by the grupoid of diffeomorphisms and isotopy classes of isotopies, let us write $S:\grap\to(\categ\downarrow2\surf)$ for the strict 2-functor which sends a corolla $\tau$ to the strict 2-category of collared surfaces with components identified with $\tau$ and collared boundary intervals compatibly identified with the legs of $\tau$. The functor to $2\surf$ is given by forgetting the extra identification data. We write $D:\forest\to(\categ\downarrow2\disc)$. These two functors define cyclic and modular operads (respectively). For a corolla $\tau$, $S(\tau)$ consists of a 2-category $\ccc$ and a functor $\FF:\ccc\to\surf$. 
%The composition $\widetilde{\psi}\circ\FF$ yields a functor which will be denoted by $S(\tau)\to2\topo$.
\\\\
We define the moduli space modular operad $\MM_\varsigma$ associated with a connected smooth sheaf $\varsigma$ by sending an object $\tau\in\ob(\grap)$ to the homotopy colimit
$$\MM_\varsigma(\tau):=\hoc_{S(\tau)}\widetilde{\varsigma}.$$
The cyclic operad $\DD_\varsigma$ of moduli space of discs is defined as
$$\DD_\varsigma(\tau):=\hoc_{D(\tau)}\widetilde{\varsigma}.$$

The following result is the analogue to Proposition 3.3.4 in \cite{Costello04}:

\begin{theorem}[Theorem 4.5.2 \cite{Giansiracusa14}] \label{jeffrey}
The derived modular envelope of the cyclic operad $\DD_\varsigma$ is weakly homotopy equivalent as a modular operad to $\MM_\varsigma$.
\end{theorem}

\section{Open equivariant TCFTs}

This section contains the equivalence between open equivariant topological conformal field theories and Calabi-Yau $A_\infty$-categories endowed with a group action. The reasonings rely, basically, on the description of the categories governing these theories and follow the steps taken in \cite{Costello07} and \cite{Ramses15}. Henceforth we will consider a field $\K$ of characteristic zero. 

%%%%%%%%%%%%%%%%%%%%%%%%%%%%%%%%%%%%%%%%%%%%%%%%%%%%%%%%%%%%
\subsection{Moduli space of surfaces with principal $G$-bundle}
%%%%%%%%%%%%%%%%%%%%%%%%%%%%%%%%%%%%%%%%%%%%%%%%%%%%%%%%%%%%

Let $\Lambda$ be a set of D-branes and $G$ a discrete group. We define a topological category $\mmm^G_\Lambda$ where:
\begin{enumerate}
\item The objects class is made of triples of the form $\alpha:= ([O], s, t)$ for $O\in \mathbb{N}$ and maps $s,t:[O]\to \Lambda$. We write $[O]$ for $\{0,\dots, O-1\}$;
\item the space of morphims $\mmm^G_\Lambda(\alpha,\beta)$ is given by moduli spaces $\mmm^G_\Lambda$ (mind the abuse of notation) of collared Riemann surfaces $\Sigma$ equipped with a principal $G$-bundle with a trivialization of the $G$-bundle over each boundary component. The surfaces $\Sigma$ has open boundary components given by disjoint parameterised intervals, embedded in the boundary and labelled by $[O]$, and free boundary components, which are intervals in $\partial\Sigma$ and must be labelled by D-branes in a way compatible with the labelling $\{s(i), t(i)\}$. %Principal $G$-bundles over a circle are in correspondence with the conjugacy classes of $G$ (Section 2 \cite{Segovia12}), therefore they can be thought of as series $(g_0,\dots, g_{C-1})$ of elements of $G$. Denote with $P_g$ the total space associated to $g\in G$, 
Write $T$ for the total space of the trivial bundle over an interval and consider $\alpha:=([O_\alpha], s, t)$ and $\beta:=([O_\beta], s',t')$; a morphism in $\mmm^G_\Lambda(\alpha,\beta)$ is given by the moduli space $\mmm^G_\Lambda$ of Riemann surfaces $\Sigma$ connecting intervals together with a principal $G$-bundle $\pi:P\to \Sigma$ such that $P_{|\text{incoming boundary}}=\bigsqcup_iT_i$ and $P_{|\text{outgoing boundary}}=\bigsqcup_jT_j$
for $i\in \{0,\dots, O_\alpha-1\}$ and $j\in\{0,\dots, O_\beta-1\}$.
Open intervals have associated an ordered pair $\{s(i), t(i)\}$ (for $0\leq i\leq O-1$), indicating where the interval begins and where it ends.
\end{enumerate}

Composition of morphisms is given by gluing surfaces and the associated bundles; the trivializations are a required ingredient when we glue two $G$-bundles together. We glue together incoming open (resp. closed) boundary components with outgoing open boundary components. Open boundary components can only be glued together if their D-brane labelling agree. Disjoint union makes $\mmm^G_\Lambda$ into a symmetric monoidal category.
\\\\
We require the positive boundary condition: Riemann surfaces are required to have at least one incoming closed boundary component on each connected component.
%\\\\
%Let us denote by $\mmm^G_{\Lambda,\text{open}}\subset \mmm^G_\Lambda$ the full subcategory whose objects are of the form $([O],\emptyset,s,t)$.

\begin{remark}
Although surfaces in $\mmm^G_\Lambda(\alpha,\beta)$ are required to be stable, we allow the following exceptional surfaces: the disc and the annulus with no open or closed boundary components and only free boundary components. These surfaces are unstable and so we define their associated moduli space to be a point.
\end{remark}

%\begin{remark}
%This definition follows the steps taken in \cite{Ramses15} with the difference that we are dealing here with triples $([O], s, t)$ instead of quadruples $([O], [C], s, t)$ due to the fact that we are focusing on open TCFTs.
%\end{remark}

%As there is no risk of confusion, let $\comp_{\K}$ denote the category of chain complexes of equivariant $\K$-modules and involution-preserving differentials. 
%\newpage
Let us consider the functor $\CC:\topo\to\comp_{\K}$ of singular chains. Applying the functor $\CC$ to $\mmm^G_\Lambda$ yields a %equivariant 
differential graded symmetric monoidal category $\goo^G_\Lambda=\CC(\mmm^G_\Lambda)$ whose objects are triples $\alpha=([O], s, t)$ and where $\homo_{\goo^G_\Lambda}(\alpha, \beta):=\CC(\homo_{\mmm^G_\Lambda}(\alpha, \beta)).$
%Let $\goo^G_\Lambda$ be the full subcategory whose objects are of the form $([O],\emptyset,s,t)$. Similarly, let $\gcc^G_\Lambda$ be the full subcategory whose objects are of the form $(\emptyset,[C],s,t)$.

\begin{remark}
Henceforth, all the categories will be differential graded symmetric monoidal categories (DGSM for short), and all the functors will be assumed to be differential graded functors.
\end{remark}

Given DGSM categories $(\aaa,\sqcup)$ and $(\bbb,\otimes)$, a symmetric monoidal functor $\FF:(\aaa,\sqcup)\to(\bbb,\otimes)$ satisfying $\FF(a\sqcup b)\cong \FF(a)\otimes \FF(b)$ for objects $a,b\in \ob(\aaa)$ is called \textit{split}. If we have quasi-isomorphisms instead, we talk about an \index{h-split functor} \textit{h-split functor}. 
\\\\
An \index{Open equivariant TCFT} \textit{open equivariant topological conformal field theory} (henceforth an open ETCFT) is a pair $(\Lambda,\FF)$ where $\Lambda$ is finite set of D-branes and $\FF$ is a h-split symmetric monoidal functor of the form $\FF:\goo^G_\Lambda\to \comp_{\K}$. 
%a morphism of open-closed ETCFTs $(\Lambda_1,\FF_1)\to (\Lambda_2, \FF_2)$ is given by a map $\Lambda_1\to\Lambda_2$ and a morphism $\FF\to \LL^\star\FF_2$, where $\LL:\goc^G_{\Lambda_1}\to \goc^G_{\Lambda_2}$ is the functor induced by the map $\Lambda_1\to\Lambda_2$; 
%an \index{Open Equivariant TCFT} \textit{open equivariant topological conformal field theory} (open ETCFT for short) is a h-split symmetric monoidal functor $\FF:\goo^G_\Lambda\to \comp_{\K}$, whereas 
%A morphism of open ETCFTs $(\Lambda_1,\FF_1)\to (\Lambda_2, \FF_2)$ is given by a map $\Lambda_1\to\Lambda_2$ and a morphism $\FF\to \LL^\star\FF_2$, where $\LL:\goo^G_{\Lambda_1}\to \goo^G_{\Lambda_2}$ is the functor induced by the map $\Lambda_1\to\Lambda_2$; 
%a \index{Closed Equivariant TCFT} \textit{closed ETCFT} is defined as a h-split symmetric monoidal functor 
%$\FF:\gcc^G_\Lambda\to \comp_{\K}.$
%Morphisms between open (resp. closed) ETCFTs are defined the same way we defined a morphism between open-closed ETCFTs.

%%%%%%%%%%%%%%%%%%%%%%%%%%%%%%%%%%%%%%%%%%%%%%%%%%%%%%%%%%%%
\subsection{Homotopical approach to moduli spaces of surfaces with principal $G$-bundle}
%%%%%%%%%%%%%%%%%%%%%%%%%%%%%%%%%%%%%%%%%%%%%%%%%%%%%%%%%%%%

%Let us begin with the observation that Costello's moduli spaces $D_{g,n,I}$ \cite{Costello04} form a cyclic operad governing $A_\infty$-algebras.
%\\\\
Vector bundles over contractible spaces are always trivial. This fact leads, if we consider $\varsigma(\Sigma)$ as the space of principal $G$-bundles over a Riemann surface $\Sigma$, to the following weak homotopy equivalence of cyclic operads: $D_{g,n,I}\simeq \DD_\varsigma$, where $D_{g,n,I}$ is the cyclic operad defined in \cite{Costello04} of Riemann surfaces with boundary, with marked points and possibly nodes on the boundary whose irreducible components are discs. The homotopy equivalence follows from Theorem 1.0.1 \cite{Costello04}. By considering their modular envelopes we obtain:
$$\mo(D_{g,n,I})\simeq \mo(\DD_\varsigma).$$

Following \cite{MacLane65}, we can think of an operad $\PP$ as a category with objects given by $\mathbb{N}$ and space of morphisms of the form $\PP(n,1)$, for $n\in \mathbb{N}$. This point of view helps to understand that $\MM_\varsigma$ is equivalent to $\mmm^G_\Lambda$ when we think of the space $\varsigma(\Sigma)$ as the space of principal $G$-bundles over a Riemann surface $\Sigma$. Theorem \ref{jeffrey} applies now in order to conclude:
$$\MM_\varsigma\simeq\mo(\DD_\varsigma)\simeq \mo(D_{g,n,I}).$$

Keeping the notations, we define the moduli space $\overline{\nnn}^G_\Lambda(\alpha,\beta)$ of stable Riemann surfaces with principal $G$-bundle as follows: its elements are stable collared Riemann surfaces with $O_\alpha$ incoming boundary components equipped with a principal $G$-bundle and a trivialization of the $G$-bundle over each component. Surfaces have $O_\beta$ outgoing boundary components. There are open marked points labelled by $[O_\alpha]$ and $[O_\beta]$ which are distributed all along the boundary components of the surfaces and replace the intervals with the trivialization. Surfaces in $\overline{\nnn}^G_\Lambda(\alpha,\beta)$ have free boundary components, which are the intervals between open marked points and those components with no marked points on them; free boundary components must be labelled by D-branes in $\Lambda$ in a way compatible with the maps $s,t:[O]\to \Lambda$. Surfaces in $\overline{\nnn}^G_\Lambda(\alpha,\beta)$ may have nodes and marked points, but we only allow boundary nodes and marked points on the boundary.
\\\\
Let us remark that, although surfaces in $\overline{\nnn}^G_\Lambda(\alpha,\beta)$ are asked to be stable, we allow the following exceptional surfaces: the disc with zero, one or two open marked points and the annulus with no open or closed points. Let $\nnn^G_\Lambda(\alpha,\beta)\subset \overline{\nnn}^G_\Lambda(\alpha,\beta)$ be the subspace of non-singular Riemann surfaces with principal $G$-bundle. These surfaces are unstable and so we define their associated moduli space a point.
\\\\
Since homotopy equivalences of topological spaces induce isomorphisms between the sets of isomorphism classes of principal $G$-bundles over those spaces, the arguments used in Section 6.1 \cite{Ramses15} apply in the equivariant setting to conclude:
\begin{proposition}
There exists a homotopy equivalence $\mmm^G_\Lambda(\alpha,\beta)\simeq \overline{\nnn}^G_\Lambda(\alpha,\beta)$ of orbi-spaces.
\end{proposition}

Let us define $\mathpzc{D}^G_\Lambda(\alpha,\beta)\subset \overline{\nnn}^G_\Lambda(\alpha,\beta)$ as the subspace consisting of Riemann surfaces whose irreducible components are all discs.
\\\\
As we are dealing with not-necessarily connected graphs, we have to take the modular envelope of $D_{g,n,I}$ to remark that it is equivalent to $\mathpzc{D}^G_\Lambda$. Furthermore, the work done in \cite{Costello04, Costello06} applies in our setting to provide $\overline{\nnn}^G_\Lambda$ with an orbi-space structure, whereas $\mathpzc{D}^G_\Lambda$ is endowed with such structure using the same arguments used in Section 5.2.1 \cite{thramses15}. All together leads to the following:

\begin{proposition}
There is a homotopy equivalence of orbi-spaces $\overline{\nnn}^G_\Lambda(\alpha,\beta)\cong \mathpzc{D}^G_\Lambda(\alpha,\beta)$.
\end{proposition}

\begin{proof}
This result follows from a concatenation of homotopy equivalences:
\[
\mathpzc{D}^G_\Lambda\simeq \mo(D_{g,n,I})\simeq \MM_\varsigma\simeq \mmm^G_\Lambda\simeq \overline{\nnn}^G_\Lambda \qedhere
\]
\end{proof}

Let $\CC^{\text{cell}}$ be a functor taking finite cell complexes $\comp_{\K}$ (Apendix A \cite{Costello07}). Applying $\CC^{\text{cell}}$ to $\mathpzc{D}^G_\Lambda(\alpha,\beta)$ one gets the following category and abuse of notation
$$\mathpzc{D}^G_{\Lambda}(\alpha,\beta):=\CC^{\text{cell}}\left(\mathpzc{D}^G_{\Lambda}(\alpha,\beta)\right)$$ 
which, due to the quasi-isomorphism $\CC^{\text{\scriptsize cell}}(X)\to \CC(X)$ (where $X$ is an orbi-cell complex, see Apendix A \cite{Costello07}), leads to the following result, which is the equivariant analogue of Lemma 6.1.7 \cite{Costello07}:

\begin{proposition} \label{quasiisos}
There is a quasi-isomorphism of DGSM categories: $\mathpzc{D}^G_{\Lambda}\cong \goo^G_\Lambda$.
\end{proposition}

%%%%%%%%%%%%%%%%%%%%%%%%%%%%%%%%%%%%%%%%%%%%%%%%%%%%%%%%%%%%
\subsection{Calabi-Yau equivariant $A_\infty$-categories} %%%%%%%%%%%%%%%%%%%%%%%%%%%%%%%
%%%%%%%%%%%%%%%%%%%%%%%%%%%%%%%%%%%%%%%%%%%%%%%%%%%%%%%%%%%%

\begin{remark}
Henceforth we will assume that the field $\K$ has characteristic zero and that, for a discrete group $G$, it is equipped with the identity map as group action: $G\times \K\stackrel{\id}{\longrightarrow} \K$ with $(g,k)\mapsto k$. 
\end{remark}

Let $\gcc$ be a monoidal $\K$-category 
%(i.e. its morphism space has the structure of a $\K$-module in a way that the compisition is $\K$-bilinear)
. An automorphism of $\gcc$ is an invertible $\K$-linear (on the morphisms) functor $\FF:\gcc\to\gcc$ such that:
\begin{enumerate}
  \item $\FF(1)=1$; 
  \item $\FF(c_1\otimes c_2)=\FF(c_1)\otimes \FF(c_2)$;
%  \item $\FF(U^\dag)=(\FF(U))^\dag$; 
  \item $\FF(f\otimes g)=\FF(f)\otimes \FF(g)$ for arrows $f,g$ in $\gcc$;
  \item There is compatibility with the morphisms defining the monoidal structure of $\gcc$.
%  \begin{enumerate}
%  \item $\FF(\alpha_{c_1,c_2,c_3})=\alpha_{\FF(c_1),\FF(c_2),\FF(c_3)}$ for $c_1,c_2,c_3\in\ob(\gcc)$;
%  \item $\FF(\lambda_c)=\lambda_{\FF(c)}$ for $c\in\ob(\gcc)$;
%  \item $\FF(\mu_c)=\mu_{\FF(c)}$ for $c\in\ob(\gcc)$,
%  \end{enumerate}
%  where $\alpha,\lambda$ and $\mu$ are the natural isomorphisms defined in Section 4.2.
\end{enumerate}

The group of automorphisms of $\gcc$ will be denoted by $\aut(\gcc)$. For our purpose we need to require the elements of $\aut(\gcc)$ to be the identity on the objects.
\\\\
%A \index{$G$-category} \textit{$G$-category} over $\K$ is a $G$-graded category $\gcc$ equipped with a homomorphism of groups $\varphi:G\to \aut(\gcc)$ such that for each $\alpha,\beta\in G$, the functor $\varphi_\alpha=\varphi(\alpha):\gcc\to\gcc$ maps $\gcc_\beta$ into $\gcc_{\alpha\beta\alpha^{-1}}$.
Let $\gcc$ be a differential graded monoidal $\K$-category equipped with a morphism $G\stackrel{\varphi}{\longrightarrow}\aut(\gcc)$ pairing an element $\alpha\in G$ with an automorphism of $\gcc$, $\varphi_\alpha:\gcc\to\gcc$. For a given morphism $f:c_1\to c_2$ in $\homo_{\gcc}$, we define the map $\varphi_\alpha(f):\varphi_\alpha(c_1)\to\varphi_\alpha(c_2)$. Such a category will be called \index{Equivariant category} \textit{equivariant category}. A functor $\EE:(\mathpzc{C}_1,\varphi)\to (\mathpzc{C}_2,\varphi')$ between equivariant categories is a functor of the underlying DG categories such that $\varphi'_\alpha\circ \EE=\EE\circ \varphi_\alpha$ for each $\alpha\in G$.

%\begin{lemma}[\cite{Turaev00}, pg. 7]
%For objects $U,V$ and $W$ in $\ob(\gcc)$ and elements $\alpha,\beta\in G$, the following identities hold:
%\begin{enumerate}
%  \item $\varphi_\alpha(U)\in \gcc$ for any $\alpha\in G$;
%  \item $\varphi_\alpha(V\otimes W)=\varphi_\alpha(V)\otimes\varphi_\alpha(W)$;
%  \item $\varphi_\alpha(V^\dag)=(\varphi_\alpha (V))^\dag$;
%  \item $\varphi_\alpha(\varphi_{\alpha^{-1}}(V))=\varphi_{\alpha^{-1}}(\varphi_\alpha(V))=\varphi_1(V)=V$;
%  \item $\varphi_\alpha(\mathbb{1})=\mathbb{1}$;
%  \item $\varphi_{\alpha\beta}(W)=\varphi_\alpha(\varphi_\beta(W))$.
%\end{enumerate}
%\end{lemma}

\begin{lemma}[pg. 7 \cite{Turaev00}]
For maps $f:c_1\to c_2,\, g:c_2\to c_3$ in $\homo_{\gcc}$ and $\alpha, \beta\in G$, the following identities hold:
\begin{enumerate}
  \item $\varphi_\alpha(f\circ g)=\varphi_\alpha(f)\circ \varphi_\alpha(g)$;
  \item $\varphi_\alpha(f\otimes g)=\varphi_\alpha (f)\otimes\varphi_\alpha(g)$;
  \item $\varphi_\alpha(\id_{c_1})=\id_{\varphi_\alpha(c_1)}$;
  \item $\varphi_{\alpha\beta}(f)=\varphi_\alpha(\varphi_\beta(f))$;
  \item $\varphi_1(f)=\varphi_\alpha(\varphi_{\alpha^{-1}}(f))=\varphi_{\alpha^{-1}}(\varphi_\alpha(f))=f$;
%  \item There is compatibility with structural morphisms (see Section 2 \cite{Turaev00}).
\end{enumerate}
\end{lemma}

%\begin{remark}
%From now on, all the functors between equivariant differential graded symmetric monoidal categories (equivariant DGSM for short) will be assumed to be differential graded functors (cf. \cite{Ramses15}).
%\end{remark}

%If $\K$ is a field endowed with the trivial action, a \index{Calabi-Yau equivariant category} \textit{Calabi-Yau equivariant category} is an equivariant category $\gcc$ with a $G$-equivariant trace map $\tr_c:\homo_{\gcc}(c,c)\to \K$ for each object $c\in \ob(\gcc)$, with $\alpha\in G$. For objects $a, b\in \ob(\gcc)$, the associated pairing
%\[
%\left.\begin{array}{cccc}\langle -,-\rangle_{c_1,c_2}: & \homo_{\gcc}(c_1,c_2)\otimes \homo_{\gcc}(c_2,c_1) & \to & \K \\ & f\otimes g & \mapsto & \tr_{c_1}(g\circ f)\end{array}\right.
%\]
%is required to be symmetric, non-degenerate and must satisfy:
%\begin{equation}\label{twisted} 
%\langle \varphi_\alpha(f),\varphi_\alpha(g)\rangle_{\varphi_\alpha(c_1),\varphi_\alpha(c_2)}=\langle f,g\rangle_{c_1,c_2}.
%\end{equation}

An \index{Equivariant $A_\infty$-category} \textit{equivariant $A_\infty$-category} is an equivariant category $\gcc$ where, for $c_1, c_2\in\ob(\gcc)$, the space $\homo_{\gcc}(c_1,c_2)$ is a finite-dimensional complex of $\K$-modules. Further, for each sequence of objects $c_0,\dots, c_n\in\ob(\gcc)$ with $n\geq 2$, there are cyclically symmetric maps
$$m_n:\homo_{\gcc}(c_0,c_1)\otimes\cdots\otimes\homo_{\gcc}(c_{n-1},c_n)\to\homo_{\gcc}(c_0,c_n)$$ 
of degree $n-2$ satisfying the usual conditions we find in \cite{Costello07}, Section 7.1, and the following equality:
$$\varphi_\alpha(m_n(f_0\otimes\cdots\otimes f_{n-1}))=m_n(\varphi_\alpha(f_0)\otimes\cdots\otimes \varphi_\alpha(f_{n-1})).$$

If for each $c\in\ob(\mathpzc{C})$ there exists an element $1_c\in\homo_{\mathpzc{C}}(c,c)$ of degree zero such that
\begin{enumerate}
\item $b_2(f\otimes 1_c)=f$ and $b_2(1_c\otimes g)=g$ for $f\in \homo_{\mathpzc{C}}(c',c)$ and $g\in \homo_{\mathpzc{C}}(c,c')$; 
\item for $0\leq i\leq n$, if $f_i\in\homo_{\mathpzc{C}}(c_i,c_{i+1})$ and $j=j+1$, then
$$b_n(f_0\otimes f_1\otimes\cdots\otimes 1_{c_j}\otimes\cdots\otimes f_{n-1})=0$$
\end{enumerate}
we say that the $A_\infty$-category $\mathpzc{C}$ is \textit{unital}.
\\\\
A functor $\EE:(\mathpzc{C}_1,\varphi)\to (\mathpzc{C}_2,\varphi')$ between equivariant $A_\infty$-categories is a functor of the underlying $A_\infty$-categories (see Section 5.1.2 \cite{Lefevre03}) such that $\varphi'_\alpha\circ \EE=\EE\circ \varphi_\alpha$ for each $\alpha\in G$. Equivariant $A_\infty$-categories and functors between them form a category.
\\\\
A \index{Calabi-Yau equivariant $A_\infty$-category} \textit{Calabi-Yau equivariant $A_\infty$-category} is an equivariant $A_\infty$-category $\gcc$ endowed with a trace map $\tr:\homo_{\mathpzc{C}}(c_1,c_1)\to \K$ compatible with the group action and a symmetric and non-degenerate pairing $\langle -,-\rangle_{c_1,c_2}$ for objects $c_1, c_2\in\ob(\gcc)$: 
\[
\left.\begin{array}{cccc}\langle -,-\rangle_{c_1,c_2}: & \homo_{\gcc}(c_1,c_2)\otimes \homo_{\gcc}(c_2,c_1) & \to & \K \\ & f\otimes g & \mapsto & \tr(g\circ f)\end{array}\right.
\]
which is required to be symmetric, non-degenerate and must satisfy:
\begin{enumerate}
\item $\langle m_{n-1}(c_0\otimes \cdots \otimes c_{n-2}), c_{n-1}\rangle=
(-1)^{(n+1)+|c_0|\sum_{i=1}^{n-1}|c_i|}\langle m_{n-1}(c_1\otimes \cdots \otimes c_{n-1}), c_0 \rangle;$
\item $\langle \varphi_\alpha(f),\varphi_\alpha(g)\rangle_{\varphi_\alpha(c_1),\varphi_\alpha(c_2)}=\langle f,g\rangle_{c_1,c_2}.$
\end{enumerate}

\subsection{Generators and relations}
%%%%%%%%%%%%%%%%%%%%%%%%%%%%%%%%%%%%%%%%%%%%%%%%%%%%%%%%%%%

%\begin{remark}
The techniques in \cite{Costello07} apply in the equivariant setting as we will not need to include any further generator, as required in \cite{Ramses15}. Observe that the disc with one incoming and one outgoing boundary component, which is equipped with a principal $G$-bundle and a trivialization being the identity on one of the boundary components and the action by $g\in G$ on the other, will play the role of the group action when describing Calabi-Yau equivariant $A_\infty$-categories. 
%\end{remark}
\\\\
%A DG category $\aaa$ is generated by some set of arrows $A$ if $\homo_{\aaa}$ has $A$ as a generating set; $\aaa$ has $R$ as a set of relations if $\homo_{\aaa}$ is given by the quotient $A/R$. We say that $\aaa$ is generated as a symmetric monoidal category by $A$ modulo $R$ if $\homo_{\aaa}$ is of the form $A/R$ and the axioms of symmetric monoidal categories are satisfied.
%\\\\
Let $\mathpzc{D}^+_{\Lambda, G}\subset \mathpzc{D}^G_{\Lambda}$ be the subcategory with the same objects but where a morphism is given by a disjoint union of discs endowed with a $G$-bundle and a trivialization on the boundaries, with each connected component having exactly one outgoing boundary marked point. For an ordered set of D-branes, with $n\geq 1$, let $[\lambda_n]:=\{\lambda_0, \dots, \lambda_{n-1}\}$ be the object in $\ob\left(\goo^G_\Lambda\right)$ with $O=n$, $s(i)=\lambda_i, t(i)=\lambda_{i+1}$ for $0\leq i\leq O-1$; we use the notation $[\lambda_n]^c:=\{\lambda_1, \dots, \lambda_{n-1},\lambda_0\}$. Let us define $(D^+(\lambda_0,\dots,\lambda_{n-1}), \varphi)$ a disc in $\mathpzc{D}^+_{\Lambda, G}$ with $n$ marked points, with D-brane labelling given by the different $\lambda_i$ and a principal $G$-bundle with a trivialization $\varphi=\{\varphi_i\}_{i\in\{0,\dots, n-1\}}$ on each marked point. All the boundary marked points are incoming except for that between $\lambda_{n-1}$ and $\lambda_0$, which is outgoing. The boundary components of the discs are compatibly oriented. 
%There is an exceptional morphism in $\mathpzc{D}^+_{\Lambda, G}$ given by a disc $D^\tau(\lambda_0,\lambda_1)$, which will be called a \textit{twisted disc}. The particularity of this disc is that, contrary to the discs $D^+(\lambda_0,\dots,\lambda_{n-1})$, it has boundary components oriented incompatibly. The role played by the twisted disc is that gluing it to a surface changes the orientation of the boundary component of the surface.
\\\\
Let $\gcc\subset \mathpzc{D}^G_{\Lambda}$ be the subcategory with $\ob(\gcc)=\ob\left(\mathpzc{D}^G_{\Lambda}\right)$ but whose arrows are not allowed to have connected components which are the disc with at most one open marked point, or the disc with two open marked incoming points or the annulus with neither open nor closed marked points. The morphisms in $\gcc$ are assumed to be not complexes but graded vector spaces.
\\\\
Let us consider a principal $G$-bundle $\xi=(D\times G, \pi, D)$ over a disc $D$. It is known that the automorphisms $\aut(D\times G)$ are defined by an element $g\in G$. To be precise:

\begin{theorem}[Theorem 1.1, Chapter 5 \cite{Husemoller93}]
Let us consider a principal $G$-bundle of the form $\xi=(B\times G, \pi, B)$. Then the automorphisms $\xi \to \xi$ over $B$ are in bijection with maps of the form $B\to G$. Particularly, an automorphism $\xi\to \xi$ has the form $h_g(b,s)=(b, g(b)s)$ for a certain map $g:B\to G$.
\end{theorem}

If we now consider two trivializations of $\xi$, say $\varphi_1$ and $\varphi_2$, we have:
$$\varphi_2(b,s)=h_g(\pi(b,s))\varphi_1(b,s),$$ 
what can be summarized by saying that any two trivializations over $D$ differ from an element of the group $G$. Clearly, a trivialization $\psi$ is defined by an element of the group $G$.
\\\\
We are interested in pairs of the form $(\xi,\varphi)$ where $\xi$ is a principal $G$-bundle over a disc and $\varphi$ is a trivialization over the boundary, what means that $\varphi$ is given by a collection $\{\varphi_1,\dots, \varphi_n\}$ with one trivialization $\varphi_i$ for each boundary component of the given disc. Observe that, as each trivialization $\varphi$ is given by an element of the group $G$, such a trivialization is determined by an $n$-tuple $\{g_1,\dots, g_n\}$ with $g_i\in G$ for $i\in\{1,\dots, n\}$. Two pairs $(\xi,\varphi)$ and $(\xi',\varphi')$ are isomorphic if there is an element $h\in G$ such that $(g'_1,\dots, g'_n)=(hg_1,\dots, hg_n)$, where $\{g'_1,\dots, g'_n\}$ is the $n$-tuple associated to $\varphi'$. Let us denote by $\overline{(\xi,\varphi)}$ the isomorphism class of $(\xi,\varphi)$.
\\
We define a \index{$G$-twisted disc} \textit{$G$-twisted disc} as a pair $\overline{(D^\mathfrak{g}(\lambda_0,\lambda_1), \varphi)}$ given by a disc $D(\lambda_0,\lambda_1)$ with two marked points endowed with a principal $G$-bundle $\xi$ and trivializations $\varphi=\{\varphi_1, g\varphi_1\}$ for $g\neq e$, the neutral element of $G$. A disc $\overline{(D(\lambda_0,\dots,\lambda_n),\varphi)}$ is \index{Untwisted disc} \textit{untwisted} if all the trivializations $\varphi_i$ agree (for $i\in\{0,\dots, n\}$), that is: if they differ by the identity $e\in G$.

\begin{proposition}[cf. Proposition 6.2.1 \cite{Costello07}]\label{generators}
Let $\overline{(D(\lambda_0,\dots, \lambda_{n-1}),\varphi)}$ be the isomorphism class of an untwisted disc 
%equipped with a principal $G$-bundle and a trivialization $\varphi=\{\varphi_i\}_{i\in\{0,\dots, n-1\}}$ on each marked point
with the marked points all incoming. The subcategory $\gcc$ is freely generated, as a symmetric monoidal category over $\ob\left(\mathpzc{D}^G_{\Lambda}\right)$, by $\overline{(D(\lambda_0,\dots, \lambda_{n-1}),\varphi)}$, for $n\geq 3$ and the discs $\overline{(D_{\text{out}}^\mathfrak{g}(\lambda_i,\lambda_j),\varphi)}$ with two outgoing marked points, subject to the relation that $\overline{(D(\lambda_0,\dots, \lambda_{n-1}),\varphi)}$ is cyclically symmetric: 
$$\overline{(D(\lambda_0,\dots, \lambda_{n-1}),\varphi)}=\pm \overline{(D(\lambda_1,\dots, \lambda_{n-1}, \lambda_0),\varphi')}.$$
\end{proposition}

\begin{remark}
Mind the fact that we need to permute the components of $\varphi$ also in the previous Proposition to get the trivialization $\varphi'$.
\end{remark}

%\begin{proof}
%The proof for this result follows the steps of Proposition 6.2.1 \cite{Costello07}. If we denote by $\eee$ a category with the same sets of generators and relations as $\gcc$, we can construct a fully faithful functor $\eee\to\gcc$, indeed: to prove that the functor is full we observe every surface in $\homo_{\gcc}(\alpha, \beta)$ can be built using disjoint unions of surfaces in $\eee$ and gluing discs. Observe that the twisted disc, as remarked above, allows us to change the orientations of the marked points, whilst the disc with two outgoing marked points turns incoming boundaries into outgoing boundaries.
%\\\\
%In order to check that $\eee\to\gcc$ is faithful, we construct an inverse functor $\gcc\to\eee$, which is the identity $\mathbb{1}_{\gcc}$ on objects . Let us consider $\Sigma\in\gcc(\alpha,\beta)$, then we can write $\Sigma=\Sigma'\circ \Upsilon$, where both $\Sigma',\Upsilon$ are surfaces in $\eee$. The surface $\Sigma'$ is composed by disjoint unions of identity maps, discs with all incoming boundaries and twisted discs; the surface $\Upsilon$ is composed by disjoint unions of identity maps, discs with two outgoing boundaries and twisted discs. This decomposition allows us to write a map $\gcc(\alpha,\beta)\to \eee(\alpha,\beta)$. We conclude that the functor $\eee\to\gcc$ is faithful.
%\end{proof}

%The following result is a corollary of Proposition \ref{generators} and its proof follows the same steps.

\begin{corollary}[cf. Lemma 6.2.2 \cite{Costello07}] \label{Mess}
The category $\mathpzc{D}^+_{\Lambda, G}$ is freely generated as a DGSM category over $\ob(\mathpzc{D}_{\Lambda,G})$ by the untwisted discs $\overline{(D^+(\lambda_0,\dots,\lambda_{n-1}),\varphi)}$ and the discs $\overline{(D^\mathfrak{g}(\lambda_i,\lambda_{i+1}), \varphi)}$ modulo the following relations: 
\begin{enumerate}
%\item for $n\geq 2$ gluing a disc $D^+(\lambda_0,\lambda_1)$ to $D^+(\lambda_0,\dots, \lambda_{n-1})$ is equivalent to gluing discs $D^+(\lambda_i,\lambda_j)$ to each boundary component of $D^+(\lambda_0,\dots, \lambda_{n-1})$ but one; \label{group}
%\begin{figure}[h]
%\centering
%\scalebox{.75}{\input{equivariant1.pdf_tex}}
%\end{figure}

\item for $n\geq 2$ and $0\leq i\leq n-2$, gluing pairs $\overline{(D^\mathfrak{g}(\lambda_i,\lambda_{i+1}),\{\psi_i,\psi_{i+1}\})}$ to the incoming boundary components of $\overline{(D^+(\lambda_0,\dots, \lambda_{n-1}),\varphi)}$ is equivalent to gluing discs $\overline{(D^+(\lambda_0,\lambda_{n-1}),\{\psi_0,\psi_{n-1}\})}$ to the only outgoing boundary component of the disc $\overline{(D^+(\lambda_0,\dots, \lambda_{n-1}),\varphi)}$; \label{prodequiv}
\begin{figure}[h]
\centering
\scalebox{.65}{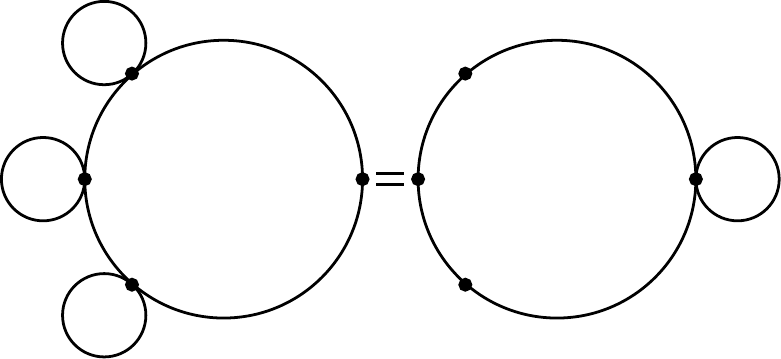}
\end{figure}

\item for $n=3$ we have the following relation: 
$$\overline{(D^+(\lambda_0,\lambda_0,\lambda_1),\varphi)}\circ \overline{(D^+(\lambda_0),\psi)}=\id_{\{\lambda_0,\lambda_1\}}=\overline{(D^+(\lambda_0,\lambda_1,\lambda_1),\xi)}\circ \overline{(D^+(\lambda_1), \zeta)};$$

%\item for $n=3:\,D^+(\lambda_0,\lambda_1,\lambda_2)\circ D_\tau^+(\lambda_1,\lambda_2)=D^+(D_\tau^+(\lambda_0,\lambda_1),D_\tau^+(\lambda_2,\lambda_0),\lambda_2);$ \label{three}

\item for $n\geq 4:\, \overline{(D^+(\lambda_0,\dots,\lambda_i,\lambda_i,\dots, \lambda_{n-1}),\varphi)}\circ \overline{(D^+(\lambda_i),\psi)}=0$.
%\item for $n\geq 4$ the composition $D^+(\lambda_0,\dots, \lambda_{n-1})\circ D_\tau^+(\lambda_{n-2},\lambda_{n-1})$ is equal to $D^+(D_\tau^+(\lambda_0,\lambda_1),\dots, D_\tau^+(\lambda_{n-2},\lambda_{n-1}),\lambda_{n-1}).$ \label{mess}
\end{enumerate}
\end{corollary}

\begin{theorem}[cf. Theorem 6.2.3 \cite{Costello07}] \label{theorem_Mess}
The category $\mathpzc{D}^G_{\Lambda}$ is freely generated, as a DGSM category over $\ob(\mathpzc{D}^G_{\Lambda})$, by $\mathpzc{D}^+_{\Lambda, G}$ and the discs $\overline{(D^\mathfrak{g}_{in}(\lambda_0,\lambda_1],\varphi)}$ with two incoming or $\overline{(D^\mathfrak{g}_{out}(\lambda_0,\lambda_1),\psi)}$ with two outgoing boundary components modulo the following relations:
\begin{enumerate}
  \item Gluing a disc with two outgoing boundary components to a disc with two incoming boundary components yields the identity;
%    \item gluing a disc $D^+(\lambda_0,\lambda_1)$ to one of the boundary components of $D_{in}(\lambda_0,\lambda_1)$ (resp. $D_{out}(\lambda_0,\lambda_1)$) is equivalent to glue it to the other boundary component; \label{trace}
  \item the untwisted disc $\overline{(D(\lambda_0,\dots,\lambda_{n-1}),\varphi)}$ is cyclically symmetric under the existing permutation isomorphism $[\lambda_n]\cong[\lambda_n]^c.$
\end{enumerate}
\end{theorem}

%\begin{proof}
%The proof follows arguments used to prove Proposition \ref{generators}.
%\end{proof}

%\begin{remark}
%Relation \ref{theorem_mess} will be needed in order to guarantee the Calabi-Yau condition in the forthcoming equivariant $A_\infty$-categories.
%\end{remark}

Let $\mathpzc{D}^+_{\Lambda, G}$ be the $\ob(\goo^G_\Lambda)\mhyphen\mathpzc{D}^+_{\Lambda, G}$-bimodule with the generators and relations stated above.
\\\\
The complexes $\mathpzc{D}^+_{\Lambda,G}$ admit a differential $d$ which is defined as follows: if $*$ denotes the gluing of the open marked points between $\lambda_i$ and $\lambda_j$:
$$d((D(\lambda_0,\dots,\lambda_{n-1}),\varphi))=\sum_{\substack{0\leq i \leq j\leq n-1 \\ 2\leq j-i}} \pm (D(\lambda_i,\dots, \lambda_j),\varphi)*(D(\lambda_j,\dots, \lambda_i),\varphi).$$

%\begin{remark}
%The signs in the previous formula for the differential are not important for our purposes; nevertheless, we point out that they depend on the orientation chosen for the cells in $\gcal_\Lambda$ of marked points on discs and annuli.
%\end{remark}

%Using the equivalence stated by Braun , \cite{thbraun12} between the categories of Klein and Riemann surfaces, we can apply the arguments used by Costello , \cite{Costello07} in order to state:

%%%%%%%%%%%%%%%%%%%%%%%%%%%%%%%%%%%%%%%%%%%%%%%%%%%%%%%%%%%%
\subsection{Calabi-Yau equivariant categories and TCFTs} %%%%%%%%%%%%%%%%%%%%%%%%%%%%%%
%%%%%%%%%%%%%%%%%%%%%%%%%%%%%%%%%%%%%%%%%%%%%%%%%%%%%%%%%%%%

Our main result states that the category of open equivariant TCFTs is quasi-isomorphic to the category of Calabi-Yau equivariant $A_\infty$-categories. Observe that for such a category $\gcc$, the elements of $\aut(\gcc)$ will be given, for each $\alpha\in G$, by the disc with one incoming and one outgoing boundary component. We will also get products $m_n$ from the generators of the categories defined in the previous sections.
\\\\
Let $\FF:\mathpzc{D}^+_{\Lambda,G}\to \comp_{\K}$ be a split symmetric monoidal functor. For each $O\in \mathbb{N}$ and D-brane labelling given by $\{s(i), t(i)\}$, with $0\leq i\leq O-1$, the following isomorphism holds:
\begin{equation} \label{isom}
\FF([O],s,t)\cong \bigotimes_{i=0}^{O-1}\FF(\{s(i), t(i)\}).
\end{equation}

Let the pair $\{s(i), t(i)\}$ correspond to the pair of D-branes $\{\lambda_i,\lambda_{i+1}\}$. We can define a category $\bbb$ with $\ob(\bbb):=\Lambda$ and $\homo_{\bbb}(\lambda_i,\lambda_{i+1}):=\FF(\{s(i), t(i)\})$. Composition of morphisms in $\bbb$ makes sense as $\FF$ is split. 
%It makes sense to talk bout $\homo_{\bbb}(\lambda_i,\lambda_j)$, with $i<j$, as we can always consider a pair $\lambda_i,\lambda_j$ as a concatenation $\lambda_i, \lambda_{i+1},\dots, \lambda_j$ of D-branes connected by the respective intervals. 
Observe that we are just associating each open boundary component to the space $\homo_{\bbb}(\lambda_i, \lambda_{i+1})$. 
%this is the analogue of what we do with Frobenius algebras and TQFTs.

\begin{lemma} \label{invinfinity}
A split symmetric monoidal functor $\FF:\mathpzc{D}^+_{\Lambda, G}\to \comp_{\K}$ is the same as a unital equivariant $A_\infty$-category $\bbb$ with set of objects $\Lambda$.
\end{lemma}

\begin{proof}
The proof follows from the isomorphism (\ref{isom}) above. Let us observe that:
% argument of this proof is exactly the same we can find , \cite{Costello07} (Lemma 7.3.1) by using the identification between Klein surfaces and Riemann surfaces with anti-analytic involution. 
%Precisely this identification contributes with the involution we equip the $A_\infty$-algebras with. 
%We recall the main points: 
\begin{enumerate}
\item The homomorphism of groups $\varphi:G\to \aut(\bbb)$ is given by
%\[
%\left.\begin{array}{cccc}\varphi: & G & \to & \aut(\gcc) \\ & \alpha & \mapsto & D_\alpha^+(\lambda_0, \lambda_1)\end{array}\right.
%\]
discs $(D^\mathfrak{g}(\lambda_0,\lambda_1),\psi)$. 
%with trivializations given by the identity on one of the marked points and a $G$-action on the other. Observe that modifying the $G$-action will produce an isomorphic pair as morphisms of principal $G$-bundles are isomorphisms (Theorem 3.2 \cite{Husemoller93}).
%describes the morphism $\gcc_\beta\to\gcc_{\alpha\beta\alpha^{-1}}$.

\item the discs $(D^+(\lambda_0,\dots, \lambda_{n-1}),\psi)$ yield the products
$$m_{n-1}:\homo_{\bbb}(\lambda_0,\lambda_1)\otimes\dots\otimes\homo_{\bbb}(\lambda_{n-2},\lambda_{n-1})\to\homo_{\bbb}(\lambda_0,\lambda_{n-1});$$
\item the differential $d$ gives the $A_\infty$-relations between the $m_n$; 
\item for $n=2,\, (D^+(\lambda_0,\lambda_1),\psi)$ yields the identity $\homo_{\bbb}(\lambda_0,\lambda_1)\to\homo_{\bbb}(\lambda_0,\lambda_1)$; %where $1\in G$ denotes the neutral element;
\item for $n=1,\, (D^+(\lambda), \psi)$ yields the unit $\K\to\homo_{\bbb}(\lambda,\lambda)$; %where $1\in G$ denotes the neutral element.
\end{enumerate}
%The definition of equivariant graded spaces asks for the identity $(x\cdot y)^\star=y^\star\cdot x^\star$, which is guaranteed in lemma \ref{Mess}.
Relation \ref{prodequiv} in Corollary \ref{Mess} shows that the products $m_n$ preserve the group action. \qedhere
\end{proof}

\begin{lemma} \label{invcalabi}
A split symmetric monoidal functor $\FF:\mathpzc{D}^G_{\Lambda}\to \comp_{\K}$ is the same as a Calabi-Yau unital equivariant $A_\infty$-category $\bbb$ with set of objects $\Lambda$.
\end{lemma}

\begin{proof}
The proof follows the same arguments of Lemma \ref{invinfinity} but now we have two more generators (see Theorem \ref{theorem_Mess}): the discs with two incoming and two outgoing marked points, which yield the map
$$\homo_{\bbb}(\lambda_0,\lambda_1)\otimes\homo_{\bbb}(\lambda_1,\lambda_0)\to \K$$
and its inverse. The extra relations on $\mathpzc{D}^G_{\Lambda}$ correspond to the cyclic symmetry condition. Relation \ref{prodequiv} of Corollary \ref{Mess} yields the identity $\langle \phi_\alpha(f),\phi_\alpha(g) \rangle=\langle f, g\rangle$. 
\end{proof}

A \index{CYUE equivariant $A_\infty$-category} \textit{Calabi-Yau unital extended equivariant $A_\infty$-category} with objects in $\Lambda$ is a h-split symmetric monoidal functor $\FF:\mathpzc{D}^G_{\Lambda}\to\comp_{\K}$. By considering $\mathpzc{D}^+_{\Lambda, G}$ instead of $\mathpzc{D}^G_{\Lambda}$ we get the concept unital extended equivariant $A_\infty$-category. If we consider split functors instead of h-split functors, we obtain the concept of unital Calabi-Yau equivariant $A_\infty$-category and the concept of unital equivariant $A_\infty$-category respectively.

\begin{proposition}
The category of Calabi-Yau unital extended equivariant $A_\infty$-categories with set of objects $\Lambda$ is quasi-equivalent to the category of open equivariant TCFTs.
\end{proposition}

\begin{proof}
Let us recall that an open equivariant TCFT is an h-split symmetric monoidal functor 
$$\KK:\goo^G_{\Lambda}\to \comp_{\K}.$$
The result follows from Lemma \ref{invcalabi} and the quasi-isomorphism between $\goo^G_\Lambda$ and $\mathpzc{D}^G_{\Lambda}$, in Proposition \ref{quasiisos}.
\end{proof}

Given categories $\mathpzc{C}$ and $\mathpzc{D}$, we define a \index{Quasi-equivalence} \textit{quasi-equivalence} as a pair of functors $\FF:\mathpzc{C}\to \mathpzc{D}$ and $\GG:\mathpzc{D} \to \mathpzc{C}$ such that the following quasi-isomorphisms hold: $\FF\circ \GG\simeq \mathbb{1}_{\mathpzc{D}}$ and $\GG\circ \FF\simeq \mathbb{1}_{\mathpzc{C}}$. 

\begin{proposition}[cf. Proposition 8.4 \cite{Ramses15}]\label{quasi-equivalence}
The following categories are quasi-equivalent:
\begin{enumerate}
\item The category of unital extended equivariant $A_\infty$-categories;
\item the category of unital equivariant $A_\infty$-categories and
\item the category of unital equivariant DG categories.
\end{enumerate}
\end{proposition}

%=========================================================================================
%=========================================================================================

\bibliographystyle{amsalpha}
\bibliography{/Users/Ramses/Documents/Universitat/Paper_Equivariant/Bibliografia}

\end{document}

%% file: equivariant2.pdf_tex
%% Creator: Inkscape inkscape 0.48.5, www.inkscape.org
%% Accompanies image file 'equivariant2.pdf' (pdf, eps, ps)
\begingroup%
  \makeatletter%
  \providecommand\color[2][]{%
    \errmessage{(Inkscape) Color is used for the text in Inkscape, but the package 'color.sty' is not loaded}%
    \renewcommand\color[2][]{}%
  }%
  \providecommand\transparent[1]{%
    \errmessage{(Inkscape) Transparency is used (non-zero) for the text in Inkscape, but the package 'transparent.sty' is not loaded}%
    \renewcommand\transparent[1]{}%
  }%
  \providecommand\rotatebox[2]{#2}%
  \ifx\svgwidth\undefined%
    \setlength{\unitlength}{224.825bp}%
    \ifx\svgscale\undefined%
      \relax%
    \else%
      \setlength{\unitlength}{\unitlength * \real{\svgscale}}%
    \fi%
  \else%
    \setlength{\unitlength}{\svgwidth}%
  \fi%
  \global\let\svgwidth\undefined%
  \global\let\svgscale\undefined%
  \makeatother%
  \begin{picture}(1,0.45896808)%
    \put(0,0){\includegraphics[width=\unitlength]{equivariant2.pdf}}%
  \end{picture}%
\endgroup%